\documentclass[11pt,reqno]{amsart}
\usepackage{amssymb,amsmath,amsthm,amsfonts}
\usepackage{mathrsfs,a4wide,dsfont}
\usepackage[utf8]{inputenc}
\usepackage{color}
\usepackage{comment}
\usepackage{graphicx}
\graphicspath{{figs/}}
\usepackage{subfigure}
\theoremstyle{plain}
\newtheorem{theorem}{Theorem}[section]
\newtheorem{lemma}[theorem]{Lemma}
\newtheorem{proposition}[theorem]{Proposition}

\newtheorem{definition}[theorem]{Definition}
\theoremstyle{definition}
\newtheorem{remark}[theorem]{Remark}
\newtheorem{conjecture}[theorem]{Conjecture}
\numberwithin{equation}{section}

\newcommand{\hs}{{\mathcal  H}}

\newcommand{\GG}{{\mathcal G}}

\newcommand{\Scal}{{\mathcal S}}
\newcommand{\R}{{\mathbb R}}

\newcommand{\N}{{\mathbb N}}

\newcommand{\Ha}{\hs}

\DeclareMathOperator{\tr}{tr}

\DeclareMathOperator{\cof}{cof}

\DeclareMathOperator{\spt}{spt}
\DeclareMathOperator{\Lip}{Lip}
\DeclareMathOperator{\Int}{int}

\newcommand{\te}{\mathfrak t}
\newcommand{\me}{\mathfrak m}

\newcommand{\K}{{\mathcal K}}

\newcommand{\e}{\varepsilon}
\newcommand{\eps}{\varepsilon}
\newcommand{\F}{\mathcal F}
\newcommand{\Gfun}{\mathcal G} % G functional
\newcommand{\evla}{\lambda_1} % first eigenvalue
\newcommand{\evmu}{\lambda_2} % second eigenvalue
\newcommand{\sphere}{\mathbb S}

    \let\TeXchi\chi
\newbox\chibox
\setbox0 \hbox{\mathsurround0pt $\TeXchi$}
\setbox\chibox \hbox{\raise\dp0 \box 0 }
\def\chi{\copy\chibox}

\title
[Mesoscale model for bilayer membranes]
{Variational analysis of a mesoscale model for bilayer membranes}
\author[L. Lussardi]
{Luca Lussardi}
\address[L.\,Lussardi]{Dipartimento di Matematica e Fisica ``N.\,Tartaglia'', Universit\`a Cattolica del Sacro Cuore, via dei Musei 41, I-25121 Brescia, Italy}
\email[]{l.lussardi@dmf.unicatt.it}
\urladdr{http://www.dmf.unicatt.it/~lussardi/}
\author[M. A. Peletier]
{Mark A. Peletier}
\address[M.\,A.\,Peletier]{Department of Mathematics and Computer Science and Institute for Complex Molecular Systems (ICMS), Eindhoven University of Technology, P.O. Box 513, 5600 MB Eindhoven, the Netherlands}
\email[]{m.a.peletier@tue.nl}
\urladdr{http://www.win.tue.nl/~mpeletie}
\author[M.\,R\"oger]
{Matthias R\"oger}
\address[M.\,R\"oger]{Fakult\"at f\"ur Mathematik, Technische Universit\"at Dortmund, Vogelpothsweg 87, D-44227 Dortmund, Germany}
\email[]{matthias.roeger@math.tu-dortmund.de}
\urladdr{http://www.mathematik.tu-dortmund.de/lsxi/roeger/}

\begin{document}

\maketitle
   
\vskip .2truecm
\begin{abstract}
We present an asymptotic analysis of a mesoscale energy for bilayer membranes that has been introduced and analyzed in two space dimensions by the second and third author ({\em Arch. Ration. Mech. Anal.} 193, 2009). The energy is both non-local and non-convex. It combines a surface area and a Monge--Kantorovich-distance term, leading to a competition between preferences for maximally concentrated and maximally dispersed configurations. Here we extend key results of our previous analysis to the three dimensional case. First we prove a general lower estimate and formally identify  a curvature energy in the zero-thickness limit. Secondly we construct a recovery sequence and prove a matching upper-bound estimate.% Our results show that the model explains the fundamental setup and stability of amphiphilic membranes. 

\small{%\dots
\vskip .3truecm
\noindent Keywor{\rm d}s: Lipid bilayers,
  curvature functionals, Monge-Kantorovich distance,
	Gamma-convergence, Hutchinson varifolds. 
}
\end{abstract}
 
\section{Introduction}
Biomembranes are remarkable structures with both fluid-like and solid-like properties. The main constituents are amphiphilic lipids, which have a `head' part that attracts water and a   `tail' part that repels it. Because of these properties such lipids organize themselves in micelle and bilayer structures, where the head parts shield the lipid tails from the contact with water. Without any covalent bonds the resulting structures resist stretching and bending but still allow the lipids to freely move in the in-plane direction.

In \cite{PeRoe09} a meso-scale model was introduced in the form of an energy for idealized and rescaled head and tail densities. This model originates from a micro-scale description in which heads and tails are treated as separate particles. The energy has two contributions: one penalizes the proximity of tail to polar (head or water) particles, and the second implements the head-tail connection as an energetic penalization.  A formal upscaling procedure leads to the following meso-scale model (for the precise definition see Section \ref{sec:setting}).

Configurations of head and tail particles are given by two rescaled density functions
\begin{align}
	u_\eps,v_\eps\,:\, \R^n\,\to\,\{0,\eps^{-1}\}\text{ with } u_\eps v_\eps\,=\, 0\text{ a.e. in }\R^n,\quad \int u_\eps\,=\,\int v_\eps \,=\, M_T, \label{eq:def-uv}
\end{align} 
where $\eps>0$ is a small parameter and $M_T>0$ is an $\eps$-independent constant characterizing a total-mass constraint. To such configurations an energy is assigned by
\begin{align} 
	\F_\eps(u_\eps,v_\eps) \,:=\, \eps\int |\nabla u_\eps| + \frac 1\eps d_1(u_\eps,v_\eps). \label{eq:def-mesofun}
\end{align}
Here $u_\eps$ corresponds to the tail density and the first term in the energy measures the boundary size of its support. This contribution arises from the interaction energy between tails and polar particles. The second term in the energy is given by the Monge-Kantorovich distance between $u_\eps$ and $v_\eps$ (see Section \ref{sec:MK} below) and is a remnant of the implicit implementation of a head-tail connection. %The choice of the $d_1$ distance is in view of the application to amphiphilic membranes not necessarily the best choice but considerably simplifies the mathematical analysis.

In~\cite{PeRoe09} two of us studied the two-dimensional case and showed that this mesoscale model has some remarkable properties. First, no structure is a-priori imposed, and quite different configurations can be compared, from a concentration of all $u_\eps$ mass in a single ball (this minimizes the first part of the energy) to arbitrarily dispersed density distributions (leading to arbitrarily small values for the distance term). It is shown in \cite{PeRoe09} that among all structures the energy prefers bilayer structures, and in this sense the model provides an explanation of the fundamental setup and stability of amphiphilic membranes. Secondly, the mathematical analysis of the model confirms that the model in fact shows the key properties of biomembranes, namely a preference for uniformly thin structures without ends (resistance to stretching and rupture) and a resistance to bending of the structure. In \cite{PeRoe09} this behavior was made precise by a passage to a macro-scale model and a rigorous Gamma-convergence result in two space dimensions. In that limit the densities concentrate on families of $W^{2,2}$-curves and a generalized Euler elastica energy is obtained for moderate-energy structures.

Besides its application to biophysics the analysis of the energy \eqref{eq:def-mesofun} also leads  to an interesting and challenging mathematical problem. The energy is a combination of a local and a non-local contribution with competing preferences for maximally concentrated and maximally dispersed configurations. In the limit we see a transition from density functions that are absolutely continuous with respect to Lebesgue measure to concentration on lower-dimensional structures. In addition there is a change of order: whereas only first derivatives appear in the meso-scale energy, the macro-scale energy is defined in terms of local curvature of the bilayer. 

In recent years we have seen an increasing interest in meso-scale models for biomembranes, see for example~\cite{Merl13,Merl13a,PaTr13,SeFr14}. Related but different models that describe an interaction of a surface area and a competing nonlinear term arise in many situations and have been studied intensively over the last years, see for example the analysis in \cite{KnMu13} and the references therein.

The main contribution of the present paper is to extend key parts of the previous analysis to three space dimensions: a general lower bound estimate, a formal derivation of the macroscopic limit $\eps\to 0$, and a rigorous upper bound that matches the formal limit. For the lower estimate we show that the mass transport problem that defines the $d_1$ distance induces a unit-length vector field $\theta$ on the boundary $S$ of the support of  $u_\eps$, related to the direction of \emph{transport rays} through the corresponding point. We show that such rays are countably Lipschitz continuous,  and we use the rays to introduce a function $M$ on $S$ that describes the mass transported through that point and that is closely related to the thickness of the $u_\eps$ layer. The key estimate~\eqref{estfond} states that the difference of $M$ from a constant and the deviation of the ray direction $\theta$ from the outer normal direction of $S$ is penalized at order $\eps^{-2}$, whereas at zeroth order we find a quadratic form in the derivatives of $\theta$. Formally this estimate gives a corresponding compactness statement and a corresponding lim-inf estimate: The head and tail densities $v_\eps$ and $u_\eps$ concentrate on two-dimensional surfaces without boundary and for such structures the limit energy is given by a surface integral over a positive quadratic form in the principal curvatures of that surface. This formal lim-inf estimate is complemented by the construction of a recovery sequence and a matching upper bound. We give a  precise description of the results in
Section \ref{sec:setting}.

Whereas the main line of argument for the lower bound is analogous to the previous analysis in \cite{PeRoe09}, important differences occur at many places of the proof and require substantial changes. The construction of a recovery sequence is much more involved than the corresponding construction for curves. We do not give a rigorous compactness and lim-inf estimate in the limit of vanishing $\eps>0$. This indeed is the most challenging part of a complete Gamma convergence statement. For a simplified situation a corresponding lim-inf estimate and compactness result will be the subject of a forthcoming paper.

In the next section we introduce the mesoscale energy and state our main results. In Section~\ref{sec:MK} we review some results of the Monge--Kantorovich optimal mass transport problem and prove the lower-bound estimate. In the last section we give the construction of a recovery sequence and prove the upper-bound estimate.

%==============
% setting
%==============
\section{Setting of the problem and main results}\label{sec:setting}

Here we give a precise description of the mesoscale energy and a precise statement of our main results. 

For any $\e>0$ and $M_T>0$ set 
$$
	\K_\e:=\bigg\{(u,v)\in BV(\R^3;\{0,\e^{-1}\})\times  L^1(\R^3;\{0,\e^{-1}\}) : \int u =\int v=M_T,\,\textrm{$uv=0$ a.e.\,in $\R^3$}\bigg\}.
$$
We denote by $J_u$  the jump set of $u$ and define the functional $\mathcal F_\e \colon L^1(\R^3)\times L^1(\R^3) \to [0,+\infty]$ by
$$
\mathcal F_\e(u,v):=\left\{\begin{array}{ll}\displaystyle \mathcal H^2(J_u)+\frac{1}{\e}d_1(u,v) & \textrm{if $(u,v)\in\mathcal K_\e$}\\
\\
+\infty & \textrm{otherwise in $L^1(\R^3)\times L^1(\R^3)$,}
\end{array}\right.
$$
where $d_1$ denotes the Monge-Kantorovich transport distance, see Section \ref{sec:MK} below. Moreover, let $\mathcal G_\e \colon L^1(\R^3)\times L^1(\R^3) \to [0,+\infty]$ be given by
$$
\mathcal G_\e:=\frac{\mathcal F_\e -2M_T}{\e^2}.
$$
Theorem~\ref{main1} below will show that $\mathcal G_\e$ is non-negative and bounds various geometric aspects of $u$ and $v$. 

By well-known results on the optimal mass transport problem, see Section \ref{sec:MK} below, there exists a unique monotone optimal transport map $T:\R^3\to\R^3$ pushing $u$ forward to $v$, and a Kantorovich potential $\phi\in \Lip_1(\R^3)$, where $\Lip_1(\R^3)$ is the set of all Lipschitz continuous functions $\R^3\to \R$ with Lipschitz constant less or equal than 1. The mass transport is along \emph{rays} with direction $\theta=\nabla\phi$.  We show in Proposition~\ref{prop:just-para} below that $\theta$ is countably Lipschitz continuous and that $\theta$ induces a function $M$ on $J_u$ that describes the amount of mass sitting on the transport ray through the respective point. 

Our first main result is the following lower bound.

\begin{theorem}\label{main1}
Let $(u,v)\in \mathcal K_\e$ and assume that $J_u=\bigcup_{j=1}^L S_j$ is the finite union of pairwise disjoint, compact, orientable surfaces $S_j$ of class $C^1$ in $\R^3$. Then there exist nonnegative measurable functions $M_j \colon S_j \to \R$, $j=1,\dots,L$ such that
\begin{align}
	M_T \,=\, \sum_{j=1}^L \int_{S_j} M_j\,d\Ha^2, \label{eq:eq-mass}
\end{align}
such that $\theta$ and the inner unit normal field $\nu$ of $\spt(u)$ on $S_j$ satisfy $\theta \cdot \nu >0$ everywhere on $\{M_j>0\}$, and such that
\begin{equation}\label{estfond}
\mathcal G_\e(u,v) \geq \sum_{j=1}^L \frac{1}{\e^2}\int_{S_j}(M_j-1)^2\,d\mathcal  H^2+\frac{1}{\e^2}\int_{S_j}\bigg(\frac{1}{\theta \cdot \nu}-1\bigg)M_j^2\,d\mathcal  H^2+\int_{S_j}\frac{M_j^4}{(\theta \cdot \nu)^3}Q(D\theta)\,d\mathcal  H^2,
\end{equation}
where  the quadratic form $Q$ is defined for an arbitrary square matrix $A$ by 
\begin{align}
	Q(A):=\frac{1}{4}(\tr  A)^2-\frac{1}{6}\tr (\cof A) \label{eq:def-Q}
\end{align}
with $\cof A$ denoting the cofactor matrix of $A$.
\end{theorem}
\begin{remark}
We will show in Lemma \ref{lem:evDtheta} below that $Q$ is a positive quadratic form in the `nontrivial' eigenvalues of $D\theta$, more precisely: for any $p\in S$ such that $D\theta(p)\in \R^{3\times 3}$ exists,
\begin{align*}
	Q(D\theta(p))&=\frac{1}{4}(\evla(p)+\evmu(p))^2-\frac{1}{6}\evla(p)\evmu(p) = \frac{1}{6}(\evla(p)+\evmu(p))^2+\frac{1}{12}(\evla(p)^2+\evmu(p)^2),
\end{align*}
where $\evla(p),\evmu(p)\in\R$ are the eigenvalues of the restriction of $D\theta(p)$ to $\theta(p)^\perp$.
\end{remark}

\begin{remark}
The conditions on the smoothness of $J_u$ may seem very stringent, as they are much stronger than the minimal regularity of jump sets of characteristic functions in $BV$. Nonetheless it is natural to make these assumptions in the context of a lower-bound estimate, since by approximation arguments we can replace $u$ and $v$ by approximations with smooth boundaries with negligible cost to $\mathcal G_\e$ (see~\cite{PeRoe09}). 
\end{remark}

In \cite{PeRoe09} a corresponding estimate for the two-dimensional case and the Gamma-convergence of $\GG_\eps$ to a generalized elastica functional has been shown. For three space dimensions a full Gamma-convergence result is much more difficult and here we only give a formal statement of the respective compactness and lim-inf estimate.% that we can expect from the lower-bound estimate \eqref{estfond}.
\begin{conjecture}\label{conj:liminf}
Given any sequence $(u_\eps,v_\eps)_{\eps>0}$ in $\mathcal K_\e$ as above with 
\begin{align*}
	\Gfun_\e(u_\eps,v_\eps)\,\leq\, \Lambda
\end{align*}
there exists a subsequence $\eps>0$ (not relabled), a finite union of (generalized) surfaces $(S_j)_{j=1,\dots,L}$ with associated even density function $\vartheta_j$ and associated (generalized) second fundamental form $\Pi$ such that  $S_j$ has empty boundary for all $j$ and such that
\begin{align}
	&u_\eps \,\to\, \sum_{j=1}^L\vartheta_j\Ha^2\lfloor S_j\quad\text{ as }\eps\to 0,\label{eq:compact}\\
	&\sum_{j=1}^L \int_{S_j}  \bigg(\frac{1}{4}H^2-\frac{1}{6}K\bigg)\,\vartheta_j d\mathcal H^2 \,\leq\, \liminf_{\e\to0}\mathcal G_{\e}(u_{\eps},v_{\eps}), \label{eq:liminf}
\end{align}
where $H,K$ are the trace and determinant of the generalized second fundamental form.
\end{conjecture}

We expect that a suitable notion of generalized surfaces is that of Hutchinson varifolds with generalized fundamental form in $L^2$~\cite{Hutc86}. $S_j$ and $\vartheta_j$ describe the support and the $2$-dimensional density of that varifold and characterize limit points of the boundaries of the support of the~$u_\eps$. If these supports are given -- as expected for moderate energy configurations -- by thin layers, then the resulting density function in fact is even.

We do not further address here the compactness and lim-inf properties but prove the existence of a recovery sequence for a given smooth limit point and show a corresponding matching upper bound estimate.
\begin{theorem}\label{thm:limsup}
Fix a smooth compact orientable hypersurface $S\subset\R^3$ without boundary such that $\Ha^2(S)=\frac{1}{2}M_T$. Then there exists a sequence $(u_\eps,v_\eps)_{\eps>0}$ in $\mathcal K_\e$ such that 
\begin{align}
	u_\eps\mathcal{L}^3 \,&\stackrel{*}{\rightharpoonup}\, 2\Ha^2\lfloor S\quad\text{ as } \eps\to 0, \label{eq:limsup-conv}\\
	\mathcal G_{\e}(u_{\eps},v_{\eps}) \,&\to\, 2\int_{S}  \bigg(\frac{1}{4}H^2-\frac{1}{6}K\bigg)\, d\mathcal H^2. \label{eq:limsup}
\end{align}
\end{theorem}
\section{Proof of the lower bound (Theorem~\ref{main1})}
%========================================
% MK distance
%========================================
\subsection{The Monge-Kantorovich distance}\label{sec:MK}

In this section we recall some basic facts about the Monge-Kantorovich distance and the optimal mass transport problem. Consider two mass distributions $u,v \in L^1(\R^3,\R^+_0)$ with compact support and 
$$
\int_{\R^3}u\,dx=\int_{\R^3}v\,dx=1.
$$
The {\it Monge-Kantorovich distance} between $u$ and $v$ is then defined by
$$
d_1(u,v):=\min \left(\int_{\R^3\times\R^3}|x-y|\,d\gamma(x,y)\right)
$$
where the minimum is taken over all Radon measures $\gamma$ on $\R^3\times\R^3$ such that
$$
\int_{\R^3\times\R^3}\varphi(x)\,d\gamma(x,y)=\int_{\R^3}\varphi(x)u(x)\,dx, \quad \int_{\R^3\times\R^3}\psi(y)\,d\gamma(x,y)=\int_{\R^3}\psi(y)v(y)\,dy
$$
for all $\varphi,\psi\in C^0_c(\R^3)$. It turns out that the Monge-Kantorovich distance is characterized by an {\it optimal mass transport problem}. We denote by $\mathcal A(u,v)$ the set of all Borel vector fields $T \colon \R^3 \to \R^3$ pushing $u$ forward to $v$, i.e.
$$
\int_{\R^3}\eta(T(x))u(x)\,dx=\int_{\R^3}\eta(y)v(y)\,dy, \quad \forall \eta \in C^0(\R^3).
$$
The following result is well known (see \cite{CaFM02} and \cite{FeMc02}).
\begin{theorem}\label{thm:opt-transport}
Let $u,v$ be as above. 
\begin{itemize}
\item[1)] There exists an {\rm optimal transport map} $T \in \mathcal A(u,v)$, i.e.\ $ T$ solves the problem
$$
\min_{\overline T \in \mathcal A(u,v)}\int_{\R^3}|x-\overline T(x)|u(x)\,dx;
$$
\item[2)] there exists a {\rm Kantorovich potential} $\phi \in {\rm Lip}_1(\R^3)$, i.e.\ $\phi$ solves the dual problem
$$
\max_{\bar \phi \in{\rm Lip}_1(\R^3)}\int_{\R^3}\bar \phi(x)(u(x)-v(x))\,dx;
$$
\item[3)] the identities 
$$
d_1(u,v)=\int_{\R^3}|x- T(x)|u(x)\,dx=\int_{\R^3}\phi(x)(u(x)-v(x))\,dx
$$
hold;
\item[4)] every optimal transport map $T$ and every Kantorovich potential $\phi$ satisfy \begin{equation}\label{rays}
	\phi(x)-\phi(T(x))=|x-T(x)|, \quad \textrm{a.e.\,$x \in \spt(u)$}.
\end{equation}
\item[5)]
$\phi$ can be chosen such that
\begin{eqnarray}
  \phi(x) &=& \min_{y\in\spt(v)}\big(\phi(y)+|x-y|\big)\quad\text{ for
  any }x\in\spt(u),\label{eq:add-prop-phi1}\\ 
  \phi(y) &=& \max_{x\in\spt(u)}\big(\phi(x)-|x-y|\big)\quad\text{ for
  any }y\in\spt(v),\label{eq:add-prop-phi2}
\end{eqnarray}
and $T$ can be chosen as the unique monotone transport map in the sense of\/~\cite{FeMc02},
\begin{gather*}
  \frac{x_1-x_2}{|x_1-x_2|}
  +\frac{T(x_1)-T(x_2)}{|T(x_1)-T(x_2)|}\,\neq\, 0\ \text{ for all
  }x_1\neq x_2\in\R^3\text{ with }T(x_1)\neq T(x_2).
\end{gather*}
\end{itemize}
\end{theorem}
A key property of the $d_1$-distance is that the associated optimal mass transport takes place along rays, that are defined as follows.
\begin{definition}\label{def:rays}
Let $\phi$ be  an optimal Kantorovich potential as above. A {\it transport ray} is a maximal line segment in $\R^3$ with endpoints $a,b\in \R^3$ such that $\phi$ has slope one on that segment, that is
\begin{align*}
	&a \in \spt(u), \quad b \in \spt(v), \quad a \neq b,\\
	&\phi(a)-\phi(b)=|a-b|,\\
	&|\phi(a+t(a-b))-\phi(b)|<|a+t(a-b)-b|, \quad \forall t>0,\\
	&|\phi(b+t(b-a))-\phi(a)|<|b+t(b-a)-a|, \quad \forall t>0.
\end{align*}
The transport set $\mathcal T$ is defined as the set of all points which lie in the relative interior of some transport ray.
\end{definition}
Some important properties of
transport rays are given in the next proposition.
\begin{proposition}[\cite{CaFM02}]\label{prop:cfm-rays}
\begin{enumerate}
\item Two rays can only intersect in a common endpoint.
\item
The set of endpoints of transport rays form a Borel set of measure zero.
\item If $z$ lies in the interior of a ray with endpoints $a\in\spt(u),b\in\spt(v)$
then $\phi$ is differentiable in $z$ with $\nabla\phi(z)\,=\,(a-b)/|a-b|$.
\end{enumerate}
\end{proposition}
%============================================================
% Parametrization by rays
%============================================================
\subsection{Parametrization by rays and mass coordinates}
\label{subsec:para}
We now fix a pair of densities $(u,v)\in\K_\eps$ such that $J_u=\bigcup_{j=1}^L S_j$ is the finite union of pairwise disjoint, compact, orientable surfaces $S_j$ of class $C^1$ in $\R^3$; we let $\Scal$ be the pairwise disjoint family of the surfaces representing $J_u$. Let us start deriving the lower estimate stated in Theorem \ref{main1}. We follow \cite{PeRoe09} and construct a suitable parametrization of the support of $u$  that allows us to characterize and estimate $d_1(u,v)$.

We fix a
Kantorovich potential $\phi\in \Lip_1(\R^3)$ for
the mass transport from $u$ to $v$ as in Proposition
\ref{thm:opt-transport} and let $\mathcal{T}$ denote the set of transport rays as defined in Definition \ref{def:rays}. 

We next introduce several quantities that relate the structure of the support of
$u, v$ to the optimal Kantorovich potential $\phi$.

\begin{definition}[Parametrization by rays]\label{def:para}
Let $S\in \Scal$  be one of the smooth, embedded, orientable surfaces that constitute the boundary of\/ $\spt(u)$. We then define
\begin{enumerate}
\item
sets $E_S,E $ of interface points that lie in the relative interior of a ray,
\begin{equation*}
  E_S := \{p\in S\cap \mathcal{T}\},\quad E\,:=\, \bigcup_{S\in\Scal} E_S,
\end{equation*}
\item
a direction field
\begin{eqnarray*}
  \theta: E\to \sphere^2, \qquad \theta(p) := \nabla\phi(p),
\end{eqnarray*}
\item
the positive and negative total ray length $L^+,L^-:E\to \R$,
\begin{eqnarray}
  L^+(p) &:=& \sup\{t>0 : \phi(p+t\theta(p))-\phi(p)\,=\,t\},\label{eq:def-L+}\\ 
  L^-(p) &:=& \inf\{t<0 : \phi(p+t\theta(p))-\phi(p)\,=\,t\},\label{eq:def-L-} 
\end{eqnarray}
\item the \emph{effective} positive ray length $l^+: E\to\R$,
\begin{align}
  l^+(p) := \sup\ \{t\in [0,L^+(p)] : p + \tau\theta(p)\in
  \Int\bigl(\spt(u)\bigr)\text{ for all }0<\tau<
  t\}.\label{eq:def-l+} 
\end{align}
\end{enumerate}
Finally we define the sets
\begin{eqnarray}\label{eq:def-E^i}
  D_S &:=& \big\{(p,t): p\in E, 0\leq t<l^+(p)\},\quad D\,:=\, \bigcup_{S\in\Scal} D_S
\end{eqnarray}
and a
map $\Psi$ which, suitably restricted, will serve as a parametrization of
$\spt(u)\cup\spt(v)$ by 
\begin{gather}
  \Psi: E\times\R\,\to\,\R^2,\quad
    \Psi(p,t) \,:=\, p + t\theta(p).\label{eq:def-psi}
\end{gather}
\end{definition}
The {\itshape effective ray length} is
introduced to obtain the injectivity of the parametrization 
in the case that a ray crosses several times the boundary
of $\spt(u)$.
The set $\{p\in E: l^+(p)>0\}$ represents the points of the boundary
where mass is transported in the `right direction' (see Figure \ref{fig:para1} and Figure \ref{fig:para2} for an illustration).
\begin{figure}
\centering
\includegraphics[width=0.8\textwidth]{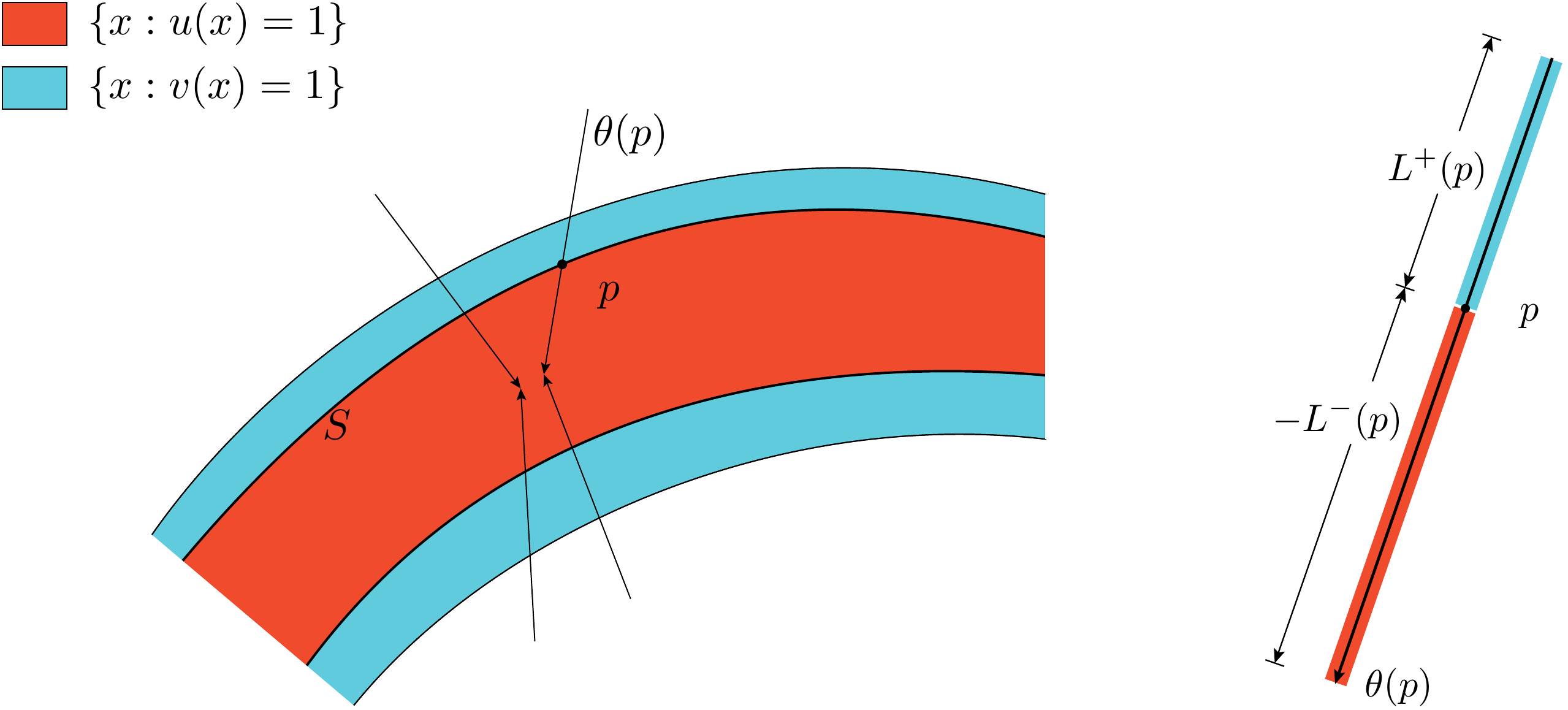}
\caption{Parametrization by rays: generic situation}
\label{fig:para1}
\end{figure}
\begin{figure}
\centering
\includegraphics[width=0.95\textwidth]{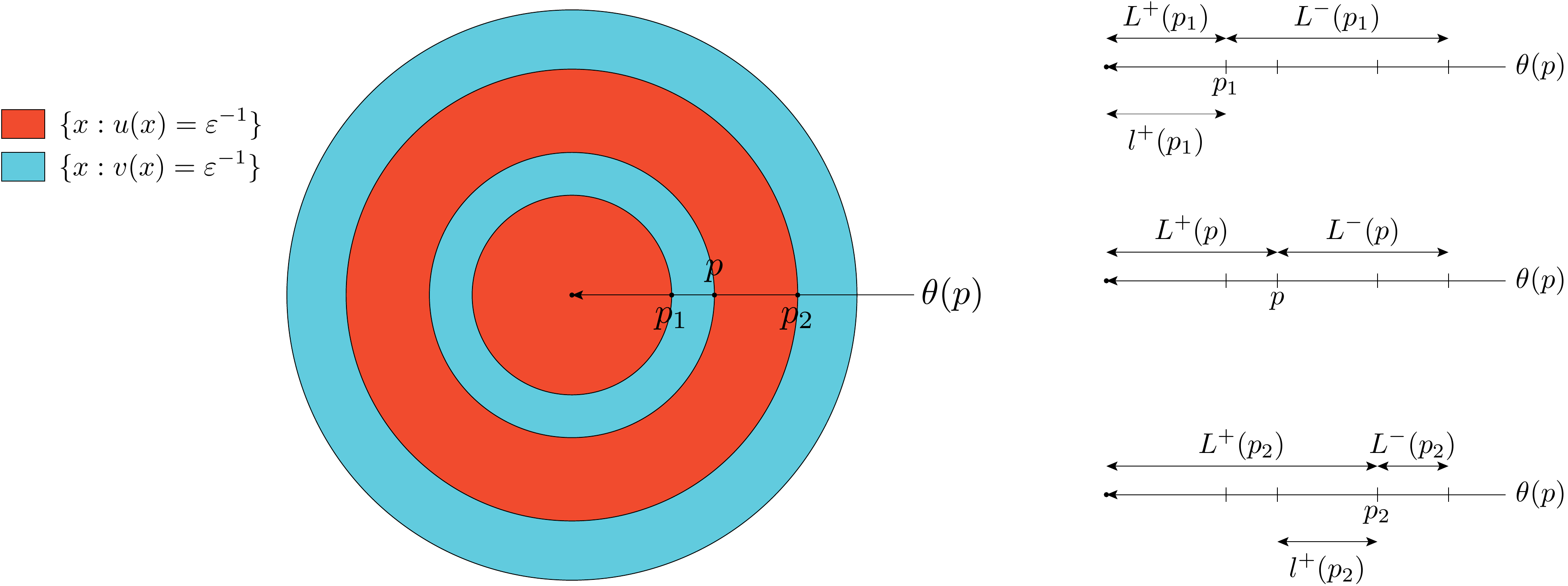}
\caption{Parametrization by rays: non-generic situation}
\label{fig:para2}
\end{figure}
The next Proposition summarizes some important properties of the quantities defined above that allow to build a suitable parametrization.
\begin{proposition}\label{prop:just-para}
\begin{enumerate}
\item\label{it:lip}
Consider for $\delta>0$
\begin{gather}
  E_\delta \,:=\, \{p\in E :
  L^+(p),|L^-(p)|\geq\delta\}. \label{eq:def-E-delta} 
\end{gather}
Then $\theta$ is Lipschitz continuous on $E_\delta$ with
\begin{eqnarray}
  \big|\theta(p_1)-\theta(p_2)\big|
  &\leq& \frac{1}{\delta}\, |p_1-p_2|\quad\text{ for all }p_1,p_2\in
  E_\delta. \label{eq:lipconst-theta} 
\end{eqnarray}
\item\label{it:Lmeas}
The positive, negative, and effective positive ray lengths $L^+, L^-,l^+ : E\to\R$
are measurable.
\item\label{it:bilip}
For any $S\in\Scal$ and almost all $\{p,t\}\in D_S$ the ray direction $\theta$ and the transformation $\Psi$ are approximately differentiable and $D\Psi(p,t): T_{(p,t)} (S\times \R)\to \R^3$ has full range.
%The parametrization $\Psi$ is locally bi-Lipschitz in $\{(p,t)\in E\times \R\,:\, \Psi(p,t)\in\mathcal{T}\}$.
\end{enumerate}
\end{proposition}
\begin{proof}
To prove \eqref{it:lip} we follow \cite[Lemma 16]{CaFM02}. Fix $x_1,x_2,z_1,z_2\in \R^3$. By the parallelogram identity we have
\begin{align}
	& |{z_1-z_2}|^2 + |x_1-x_2|^2 - |(z_1+x_1)-(z_2+x_2)|^2 \notag\\
	=\,& |z_1-x_1|^2 + |z_2-x_2|^2 - |z_1-x_2|^2 - |x_1-z_2|^2. \label{eq:lip0}
\end{align}
Next fix $S\in\Scal$ and $p_1,p_2\in E_S$ with $L^+(p_i),|L^-(p_i)|\geq \delta$, $i=1,2$ and let 
\begin{align*}
	z_i \,=\, p_i + \delta\theta(p_i),\quad  x_i \,=\, p_i - \delta\theta(p_i),\qquad i=1,2.
\end{align*}
Then we obtain for the potential $\phi$
\begin{align*}
	\phi(z_i) \,=\, \phi(p_i) + \delta,\quad  \phi(x_i) \,=\, \phi(p_i) - \delta,\qquad i=1,2.
\end{align*}
It follows that
\begin{align}
	4\delta^2 |\theta(p_1)-\theta(p_2)|^2 \,&=\, |(z_1-x_1)-(z_2-x_2)|^2, \label{eq:lip3}\\
	|(z_1-z_2)+(x_2-x_1)|^2 \,&=\, 2|z_1-z_2|^2 + 2|x_2-x_1|^2 - |(z_1-z_2)-(x_2-x_1)|^2, \label{eq:lip4}\\
	|z_1-x_1|\,&=\, 2\delta \,=\, \phi(z_1)-\phi(x_2) -\big(\phi(p_1)-\phi(p_2)\big), \notag\\
	|z_2-x_2|\,&=\, 2\delta \,=\, \phi(z_2)-\phi(x_1) -\big(\phi(p_2)-\phi(p_1)\big)\notag
\end{align}
and therefore
\begin{align}
	|z_1-x_1|^2 + |z_2-x_2|^2 \,&=\, \big(\phi(z_1)-\phi(x_2)\big)^2 + \big(\phi(z_2)-\phi(x_1)\big)^2 \notag\\
	&-2\big(\phi(p_1)-\phi(p_2)\big)\big(\phi(z_1)-\phi(x_2)-\phi(z_2)+\phi(x_1)\big) + 2 \big(\phi(p_1)-\phi(p_2)\big)^2 \notag\\
	&=\, \big(\phi(z_1)-\phi(x_2)\big)^2 + \big(\phi(z_2)-\phi(x_1)\big)^2 -2\big(\phi(p_1)-\phi(p_2)\big)^2 \notag\\
	&\leq\, |z_1-x_2|^2 + |z_2-x_1|^2 - 2\big(\phi(p_1)-\phi(p_2)\big)^2. \label{eq:lip7}
\end{align}
We then deduce
\begin{align}
	4\delta^2|\theta(p_1)-\theta(p_2)|^2 \,&\overset{\eqref{eq:lip3}}{=}\, |z_1-x_1 -(z_2-x_2)|^2 \notag\\
	&\overset{\eqref{eq:lip4}}{=}\,2|z_1-z_2|^2 + 2|x_2-x_1|^2 - |z_1-z_2 -(x_2-x_1)|^2\notag\\
	&\overset{\eqref{eq:lip0}}{=}\,4|z_1+x_1 -(z_2+x_2)|^2 - 4|z_1-z_2 -(x_2-x_1)|^2 \notag\\
	&\qquad +2\big(|z_1-x_1|^2 + |z_2-x_2|^2 -|z_1-x_2|^2 -|x_1-z_2|^2\big) \notag\\
	&\overset{\eqref{eq:lip7}}{\leq}\, 4|z_1+x_1 -(z_2+x_2)|^2 -4\big(\phi(p_1)-\phi(p_2)\big)^2 \notag\\
	&\overset{\phantom{\eqref{eq:lip0}}}{=}\, 4|p_1-p_2|^2 -  4\big(\phi(p_1)-\phi(p_2)\big)^2. \label{eq:lip8}
\end{align}
This implies the first claim.\\
For the proof of \eqref{it:Lmeas} we refer to \cite[Lemma 7.5]{PeRoe09} and \cite{CaFM02}.\\
As \eqref{it:bilip}: using a countability argument and the measurability of $l^+$ we see that it is enough to prove the claim for all $S\in\Scal$ and almost all $(p,t)\in D_S$ with $l^+(p)>\delta$ such that the set $\{\tilde{p} \in E_S\,:\,l^+(\tilde p)>\delta\}$ has full density at $p$. Since $\theta$ is Lipschitz continuous on this set we further may assume that $\theta$ is approximately differentiable in $p$. We then obtain that $\Psi$ is approximately differentiable in $(p,t)$ for any $t\in (0,l^+(p))$. Let now $\sigma = \frac{1}{2}\min\{L^+(p)-t,t-L^-(p)\}>0$. If the claim is false there exists $\tau\in T_pS$ and a sequence $s_k\to 0$ $(k\to\infty)$ and a sequence $p_k\to p$  with
\begin{align*}
	\frac{p_k-p}{s_k}\,\to\, \tau, \qquad
	\left| \frac{\Psi(p_k,t+s_k)-\Psi(p,t)}{s_k}\right| \,\to\, 0.
\end{align*}
By the measurability of $L_+,L_-$ we may assume $\min\{L^+(p_k)-(t+s_k),t+s_k-L^-(p_k)\}>\sigma$ for all $k\in\N$. Now we deduce from \eqref{eq:lip8} with $p_1=\Psi(p,t), p_2= \Psi(p_k,t+s_k), \delta = \sigma$ that
\begin{align*}
	0\,=\,& \lim_{k\to\infty} 2\left| \frac{\Psi(p_k,t+s_k)-\Psi(p,t)}{s_k}\right|\\
	\geq\, & \lim_{k\to\infty} \left| \frac{p_k -p}{s_k} + (t+s_k)\frac{\theta(p_k)-\theta(p)}{s_k} + \theta(p)\right| + \sigma \lim_{k\to\infty} \left|\frac{\theta(p)-\theta(p_k)}{s_k}\right|.
\end{align*}
We deduce that the second term has to vanish, i.e. $\frac{\theta(p)-\theta(p_k)}{s_k}\to 0$, which implies
\begin{align*}
	0 \,=\, \tau + \theta(p).
\end{align*}
Since $\tau \in T_pS$ but on the other hand $\theta(p)\cdot\nu(p)>0$ this is a contradiction.
\end{proof}
It is often convenient to generate from $D\theta(p)$ an orthonormal basis of eigenvectors. In particular, we can represent $Q(D\theta)$ in terms of the corresponding eigenvalues.
\begin{lemma}\label{lem:evDtheta}
For almost all $p\in E$, $D\theta(p)$ is diagonalizable, and there exists a positively oriented orthonormal basis $\{v_1,v_2,\theta(p)\}$ of eigenvectors with $\det(v_1,v_2,\theta(p))=1$ and eigenvalues $\evla,\evla$ such that
\begin{align*}
	D\theta(p) v_1 = \evla v_1,\quad D\theta(p) v_2 = \evmu v_2, \quad [D\theta(p)] \theta(p) = 0.
\end{align*}
Moreover, we have
\begin{align}
	\tr D\theta \,&=\, \evla + \evmu,\qquad \tr (\cof D\theta) \,=\, \evla\evmu, \notag\\
	Q(D\theta) \,&=\, \frac{1}{4}(\evla+\evmu)^2-\frac{1}{6}\evla\evmu = \frac{1}{6}(\evla+\evmu)^2+\frac{1}{12}(\evla^2+\evmu^2). \label{eq:evDtheta}
\end{align}
\end{lemma}
\begin{proof}
The assertions follow easily from the symmetry of $D\theta$, the fact that $|\theta|=1$, and direct computations.
\end{proof}
We proceed by building a parametrization and justifying a transformation
formula. For each $S$ as above, let $\Psi$ be as in \eqref{eq:def-psi}. Restricting this map 
suitably we obtain a parametrization of $\spt(u)$ which is
essentially one-to-one. 
\begin{proposition}\label{prop:parametrization}
For each $S\in \Scal$ denote by $E_S, \Psi_S, L^+_S, D_S$ etc.\ the respective quantities as defined above.
The restrictions $\Psi_S: D_S\to\R^2$ give, up to a Lebesgue
null set, an injective map onto $\spt(u)$:
for almost all $x\in \spt(u)$ there exists a unique
$S\in\Scal$ and a unique 
$(p,t)\in D_S$ such that $\Psi_S(p,t)=x$.
\end{proposition}
\begin{proof} The proof is a (straightforward) adaptation of \cite[Prop. 7.9]{PeRoe09}.%
\end{proof}
Next we give a characterization of the Jacobian of $\Psi$.
\begin{proposition}\label{prop:Jac}
For $S\in\Scal$ and $\Psi$ as above, we have for almost every $(p,t)\in D_S$,
\begin{equation}\label{trJ}
	J\Psi(p,t) \,=\, \theta(p) \cdot \nu(p)\,\bigl(1+t\,\tr D\theta(p)+t^2\tr(\cof D\theta(p))\bigr)\,>\, 0.
\end{equation}
\end{proposition}
\begin{proof}
Fix $p\in S$, $t\in \R$ such that the approximate differential $D\Psi(p,t)$ exists. Choose an orthonormal basis $\{\tau_1,\tau_2\}$ of $T_pS$, and let $\{v_1,v_2,\theta\}$ be the orthonormal basis of $\R^3$ of eigenvectors of $D\theta_p$ with $\det(v_1,v_2,\theta)=1$, as in Lemma \ref{lem:evDtheta}. We then consider the orthonormal basis $\{(\tau_1,0),(\tau_2,0),(0,1)\}$ of $T_{(p,t)} (S\times \R)$ and $\{v_1,v_2,\theta\}$ of $\R^3$, and we can assume that  $\det(\tau_1,\tau_2,\nu(p))=1$. For the representation of $D\Psi(p,t): T_{(p,t)} (S\times \R)\to \R^3$ with respect to theses bases we then obtain the matrix 
\begin{align*}
	\begin{pmatrix} 
		(1+t\evla)\tau_1\cdot v_1 & (1+t\evmu)\tau_1\cdot v_2 & \tau_1\cdot\theta_p \\
		(1+t\evla)\tau_2\cdot v_1 & (1+t\evmu)\tau_2\cdot v_2 & \tau_2\cdot\theta_p \\
		0 & 0 & 1
	\end{pmatrix}
	.
\end{align*}
The Jacobian is then given by
\begin{align*}
	 J\Psi(p,t) \,&=\, \big|\bigl(1 + t(\evla+\evmu) +t^2\evla\evmu\bigr)\bigl((\tau_1\cdot v_1)(\tau_2\cdot v_2)-(\tau_1\cdot v_2)(\tau_2\cdot v_1)\bigr)\bigr|\\
	 &=\, (\theta\cdot\nu)\,\big|1 + t\tr D\theta +t^2 \tr\cof D\theta\big|,
\end{align*}
where we have used Lemma \ref{lem:evDtheta} and the fact that $(\tau_1\cdot v_1)(\tau_2\cdot v_2)-(\tau_1\cdot v_2)(\tau_2\cdot v_1)= (\tau_1\times \tau_2)\cdot (v_1\times v_2)= \nu\cdot \theta$. 

By Proposition \ref{prop:just-para} we know that $\det J\Psi(p,t) \neq 0$. On the other hand $J\Psi(p,t)$ depends continuously on $t$ and $1 + t\tr D\theta +t^2 \tr\cof D\theta=1$ for $t=0$. We therefore deduce that $1 + t\tr D\theta +t^2 \tr\cof D\theta>0$ for all $t\in (0,l^+(p))$. This concludes the proof of the Proposition.
\end{proof}
It is often more convenient to work not in {\itshape length coordinates} $t\in (L_-(p),L_+(p))$, but rather in {\itshape mass coordinates} which are defined as follows.
\begin{definition}\label{def:mass-coordinates}
For $p\in S$, $S\in \Scal$ we define a map
$\mathfrak m: S\times \R\to \R$, also written $\mathfrak{m}_p:\R\to \R$, and a map $M: S\to\R$ by
\begin{align}\label{eq:def-mass}
  \mathfrak{m}(p,t) &:=
  \begin{cases}
    \int_0^t J\Psi(p,\tau)\frac{1}{\eps}\,d\tau &\text{ if }l^+(p)>0,\\
    0 &\text{ otherwise.}
  \end{cases}\\
  \label{eq:def-Ms}
  M(p) &:= \mathfrak{m}(p,l^+(p)).
\end{align}
\end{definition}
By \eqref{trJ} and Lemma \ref{lem:evDtheta} we have
\begin{align}
	\mathfrak{m}(p,t) \,&=\, \theta(p) \cdot \nu(p)\frac{1}{\eps}\bigg(t+\frac{t^2}{2}(\evla(p)+\evmu(p))+\frac{t^3}{3}\evla(p)\evmu(p)\bigg).\label{eq:def-mass2}
\end{align}
We obtain next that the change of variables between length- and mass-coordinates is one-to-one. We also give an important estimate for the inverse.
\begin{proposition}\label{prop-mass}
Fix $p\in S$ with $l^+(p)>0$. Then the map $\mathfrak{m}_p$ is strictly monotone on $\R$. The inverse $\mathfrak t_p$ satisfies for all $m \in \R$
\begin{equation}\label{lemmatec}
	\te_p(m)-\te_p(-m) \,\geq\, \frac{2\e}{\theta(p)\cdot \nu(p)}m+\frac{4\e^3Q(D\theta(p))}{(\theta(p)\cdot \nu(p))^3}m^3.
\end{equation}
\end{proposition}
\begin{proof}
In what follows for the sake of simplicity we drop the dependence on $p$. By a straightforward computation we get
\begin{align}\label{dert1}
	\te'(m) \,&=\, (\me'(\te(m)))^{-1},\\
	\label{dert2}
	\te''(m) \,&=\, -(\me'(\te(m)))^{-3}\me''(\te(m)),\\
	\label{dert3}
	\te'''(m) \,&=\, 3(\me'(\te(m)))^{-5}(\me''(\te(m)))^2-(\me'(\te(m)))^{-4}\me'''(\te(m)),\\
	\label{dert4}
	\te^{\rm iv}(m) \,&=\, -15(\me'(\te(m)))^{-7}(\me''(\te(m)))^3+10(\me'(\te(m)))^{-6}\me''(\te(m))\me'''(\te(m)),\\
	\label{dert5}
	\te^{\rm v}(m) \,&=\, 5(\me'(\te(m)))^{-9}(21(\me''(\te(m)))^4-21(\me''(\te(m)))^2\me'''(\te(m))\me'(\te(m))\\
			&\qquad +2(\me'(\te(m)))^2(\me'''(\te(m)))^2).
\end{align}
Therefore, using \eqref{dert1} and the  definition of $\me$ we deduce that
\begin{align}
	\te(0)&=0, \quad \te'(0)=\frac{\eps}{\theta \cdot \nu},\quad
	\te''(0) =-\eps^2 \frac{\evla+\evmu}{(\theta\cdot\nu)^2},\quad 
	\te'''(0) = \eps^3(\theta\cdot\nu)^3\big(3(\evla+\evmu)^2-2\evla\evmu\big). \label{eq:dev-t-0}
\end{align}
We further claim that
\begin{equation}\label{convder3}
	r \mapsto \te'''(r) \quad \textrm{is convex.} 
\end{equation}
In fact we check that $\te^{\rm v}(m) \geq 0$ for each $m$. Since $\me'>0$ by Proposition \ref{trJ} we find that this is equivalent to
\begin{align*}
	0\,&\leq\, 21\big(\evla(1+t\evmu)+\evmu(1+t\evla)\big)^4 - 21\big(\evla(1+t\evla)+\evmu(1+t\evla)\big)^22\evla\evmu(1+t\evla)(1+t\evla)\\
	&\qquad\qquad + 4\evla^2\evmu^2(1+t\evla)^2(1+t\evla)^2 \\
	&=\, 21(\xi +\eta)^4 -42\xi\eta(\xi+\eta)^2 + 8\xi^2\eta^2\\
	&=\, 21(\xi^2+\eta^2)(\xi +\eta)^2 + 8\xi^2\eta^2,
\end{align*}
where we have substituted $\xi =\evla(1+t\evmu), \eta=\evmu(1+t\evla)$ and which shows \eqref{convder3}.
%\meta{Strange enough: $t'''(r)=Q(\xi,\eta)$... I do not know what this means...}

Using Taylor's formula we therefore obtain
\begin{align}
	\te(m) - \te(-m) \,&=\,  \frac{\e m}{\theta \cdot \nu}-\frac{\e^2(\evla+\evmu)m^2}{2(\theta \cdot \nu)^2}+\frac{1}{2}\int_0^m(m-r)^2\te'''(r)\,dr - \notag\\
&\qquad - \Big(-\frac{\e m}{\theta \cdot \nu}-\frac{\e^2(\evla+\evmu)m^2}{2(\theta \cdot \nu)^2}-\frac{1}{2}\int_0^m(m-r)^2\te'''(-r)\,dr\Big) \notag\\
	&=\,\frac{2\e m}{\theta \cdot \nu}+\frac{1}{2}\int_0^m(m-r)^2(\te'''(r)+\te'''(-r))\,dr. \label{concl1}
\end{align}

Next \eqref{convder3} implies, using also \eqref{dert3},
$$
	\frac{1}{2}(\te'''(r)+\te'''(-r))\geq \te'''(0)=\frac{\e^3(3(\evla+\evmu)^2-2\evla\evmu)}{(\theta \cdot \nu)^3}
$$
and then \eqref{concl1} easily gives
$$
	\te(m)-\te(-m)\geq\frac{2\e m}{\theta \cdot \nu}+\frac{\e^3((\evla+\evmu)^2-\frac{2}{3}\evla\evmu)m^3}{(\theta \cdot \nu)^3}=\frac{2\e m}{\theta \cdot \nu}+\frac{4\e^3Q(D\theta)m^3}{(\theta \cdot \nu)^3}
$$
which is what we wanted to prove.
\end{proof}

We finally obtain the following two transformation formulae.
\begin{proposition}\label{prop:trafo-formula}
Let $\Scal$ and $D_S$, $S\in\Scal$ be as above. 
Then for all $g\in L^1(\R^3)$,
\begin{align}\label{eq:trafo}
  \int {}&g(x)u(x)\,dx \notag\\
  =\,&
  \sum_{S\in\Scal} \int_{D_S}\int_0^{l^+(p)} g(\Psi_S(p,t))\theta(p) \cdot \nu(p)\big(1+t\,\tr D\theta(p)+t^2\tr (\cof D\theta(p))\big)\,dt\,d\Ha^2(p)\\
    =\, &\sum_{S\in\Scal} \int_S \int_0^{M_S(p)} g(\Psi_S(p,\te_S(p,m)))\,dm\,d\Ha^2(p)	\label{eq:trafo2}
\end{align}
holds.
In particular, the total mass of $u$ is given by
\begin{align}
  \int u(x)\,dx \,&=\,
    \sum_{S\in\Scal} \int_S M_S(p)\,d\Ha^2(p) .\label{eq:tot-mass}
\end{align}
\end{proposition}
\begin{proof}
We deduce from the generalized transformation formula~\cite[Remark 5.5.2]{AmGS05} that
\begin{align*}
  \int_{D_S} &g(\Psi_S(p,t))\frac{1}{\eps}\theta(p) \cdot \nu(p)\big(1+t\,\tr D\theta(p)+t^2\tr (\cof D\theta(p))\big)\,dt\,d\Ha^2(p)\\
  &=
  \int_{\Psi_S(D_S)} g(x)u(x)\,dx. 
\end{align*}
Summing these equalities over $S\in\Scal$ we deduce by Proposition
\ref{prop:parametrization} that \eqref{eq:trafo} holds. Since
\begin{align*}
	\frac1\e\theta(p) \cdot \nu(p)\big(1+t\,\tr D\theta(p)+t^2\tr (\cof D\theta(p))\big)
	\,=\, \partial_t \mathfrak{m}(p,t)
\end{align*}
\eqref{eq:trafo2} follows.
\end{proof}

As in the two-dimensional case \cite{PeRoe09} we can decompose the
optimal mass transport problem of transporting $u$ to $v$ into
one-dimensional mass transport problems on single rays, which allows for an explicit characterization of the unique monotone optimal transport. We therefore deduce the following estimate of the transport map.
\begin{lemma}\label{lem:eff-mass-balance}
Let $S\in\Scal$.
Then for $\Ha^2$-almost all $p\in E$ with $l^+(p)>0$ we
obtain that
\begin{gather}
  \big|\Psi\big(p,\te(p,m)\big)-
  T\big(\Psi\big(p,\te(p,m)\big)\big)\big|
  \,\geq\, \te(p,m)-\te(p,m-M(p))
  \label{eq:esti-t-diffs} 
\end{gather}
for all $0<m<M(p)$.
\end{lemma}
\begin{proof}
We associate to
the optimal transport map $T$ from $u$ to $v$ and to $p\in E_S$ with $l^+(p)>0$ an interval $I(p)\subset\R$ and
measures $f_p^+,f_p^-$ on $I(p)$,
\begin{align*}
  I(p) &:=
  \Big(\mathfrak{m}_p\big(L^-(p)\big),
  \mathfrak{m}_p\big(L^+(p)\big)\Big),\\  
  df_p^+ &:= u\big(\Psi(p,\te(p,m))\big)\,dm,\qquad
  df_p^- := v\big(\Psi(p,\te(p,m))\big)\,dm.
\end{align*}
We define a map
\begin{eqnarray*}
  \hat{T}:I(p) \cap\spt(f_p^+) &\to& I(p)\cap\spt(f_p^-),
\end{eqnarray*}
by the equation 
\begin{gather}\label{eq:def-s-hat}
  T\big(\Psi(p,\te(p,m))\big) \,=\,
  \Psi\big(p,\te(p,\hat{T}(m))\big)\quad\text{ for }m\in I\cap\spt(f_p^+).
\end{gather}
Then one shows as in \cite[Prop. 7.15]{PeRoe09} that for $\Ha^2$-almost all $p\in E$ with $l^+(p)>0$
the map $\hat{T}$
is the unique monotone transport map pushing $f_p^+$ forward to $f_p^-$. Following the arguments in \cite[Lemma 7.16]{PeRoe09} we then deduce the estimate \eqref{eq:esti-t-diffs}.
\end{proof}

\begin{proof}[Proof of Theorem \ref{main1}]
We are now ready to prove the lower bound estimate and Theorem \ref{main1}. Let $T$ be the optimal transport map and $\phi\in \Lip_1(\R^3)$
be an 
optimal Kantorovich potential for the mass transport from $u$ to
$v$ as in Theorem \ref{thm:opt-transport}. By \eqref{eq:def-l+} and by \eqref{eq:def-Ms} we have $\nu\cdot\theta>0$ on $\{M>0\}$. Moreover, \eqref{eq:tot-mass} proves \eqref{eq:eq-mass}. It thus remains to prove the lower bound \eqref{estfond}.

Using Proposition \ref{prop:trafo-formula} and \eqref{eq:trafo2}
we can rewrite $d_1(u,v)$ as follows:
\begin{align}
	d_1(u,v) \,&=\, \int_{\R^3}|z-T(z)|u(z)\,dz \notag\\
	&=\,\sum_{S\in\Scal} \int_S\int_0^{M_S(p)}|\Psi_S(p,\mathfrak t_p(m))-T(\Psi_S(p,\mathfrak t_p(m)))|\,dm\,d\mathcal H^2(p). \label{eq:d1-1}
\end{align}
Next we fix $S\in\Scal$, drop for the moment the index $S$, and use Proposition \ref{prop-mass} and Lemma~\ref{lem:eff-mass-balance} to compute that
\begin{align}
	\int_S\int_0^{M(p)}&|\Psi(p,\mathfrak t_p(m))-T(\Psi(p,\mathfrak t_p(m)))|\,dm\,d\mathcal H^2(p)\notag\\
	\geq\,& \int_S \int_0^{M(p)}(\mathfrak t_p(m)-\mathfrak t_p(-m))\,dm\,d\mathcal H^2(p) \notag\\
	\geq \,&\int_S \int_0^{M(p)} \bigg(\frac{2\e}{\theta(p) \cdot \nu(p)}m+\frac{4\e^3Q(D\theta(p))}{(\theta(p) \cdot \nu(p))^3}m^3\bigg)\,dm\,d\mathcal H^2(p) \notag\\
	=\,&\int_S \bigg(\e\frac{M^2(p)}{\theta(p) \cdot \nu(p)}+\e^3\frac{M^4(p)}{(\theta(p) \cdot \nu(p))^3}Q(D\theta(p))\bigg)\,d\mathcal H^2(p). \label{eq:d1-2}
\end{align}
This implies
\begin{align}
	\mathcal G_\e(u,v) \,&=\, \sum_{S\in\Scal}\frac{1}{\e^2}\int_S 1\, d\Ha^2(p)+\frac{1}{\e^3}d_1(u,v)-\frac{2M_T}{\e^2} \notag\\
	&=\, \sum_{S\in\Scal}\frac{1}{\e^2}\int_S \big(1 - 2M(p)\big) d\Ha^2(p)+\frac{1}{\e^3}d_1(u,v) \notag\\
	&\geq\, \sum_{S\in\Scal}\frac{1}{\e^2}\Big( \int_S(M(p)-1)^2\,d\Ha^2(p)-\int_S {M(p)^2} \,d\mathcal H^2(p)\Big) +\frac{1}{\e^3}d_1(u,v)\notag\\
	&\geq\, \sum_{S\in\Scal}\Big[\frac{1}{\e^2}\int_S(M(p)-1)^2\,d\Ha^2(p)+\frac{1}{\e^2}\int_S\bigg(\frac{1}{\theta(p) \cdot \nu(p)}-1\bigg)M^2(p)\,d\Ha^2(p) + \notag\\
	&\qquad+\int_S\bigg(\frac{M^4(p)}{(\theta(p) \cdot \nu(p))^3}Q(D\theta(p))\bigg)\,d\Ha^2(p)\Big],
\end{align}
which gives \eqref{estfond}.
\end{proof}

%
%=======================================
% limsup
% ==========================================
\section{Construction of a recovery sequence}
In this section we prove Theorem \ref{thm:limsup}. We therefore consider a fixed smooth compact orientable hypersurface $S$ without boundary. By a rescaling argument we can  restrict ourselves to the case $M_T=1$, and in particular we assume that $\Ha^2(S)=\frac{1}{2}$. As in the two-dimensional case \cite{PeRoe09} the idea is to choose the boundaries of the supports of $u_\eps$ and of $v_\eps$, respectively, as  surfaces parallel to $S$, and to define a transport map with transport rays that are in the normal direction of the given surface $S$. However, the construction is more difficult in three dimensions, and  in particular we can not choose the mass functions constant $\mathfrak{m}_p(\cdot,\eps)=1$ but need some higher order corrections in $\eps$. Furthermore, in two dimensions the construction was described in terms of the parallel curves of distance $\pm\eps$ from $S$ (characterizing the boundary of the $u_\eps$ layer) and sharp estimates were achieved for both curves separately. In three dimensions we describe the construction completely in terms of $S$ and use the cancellation of contributions from the transport through the two surfaces parallel to $S$.

\subsection{General construction}
  
Consider for $\eps>0$ sufficiently small real numbers $\ell_\pm(\e)$ and smooth maps $L_\pm(\e,\cdot)$ on $S$, to be chosen later, with
\begin{align*}
	L_-(\eps,\cdot) \,<\, \ell_-(\eps)\,<\,0\,<\,\ell_+(\eps)\,<\,L_+(\eps,\cdot),
\end{align*}
with $L_\pm(\eps,\cdot),\ell_\pm(\eps)\to 0$ as $\e\to 0^+$. 

Let $\nu$ be a choice of a smooth unit normal vector field on $S$ and consider the mapping
\begin{gather}
  \Psi: S\times (-\delta_0,\delta_0) \,\to\,\R^2,\quad
    \Psi(p,t) \,:=\, p + t\nu(p). \label{eq:def-psi-limsup}
\end{gather}
For $\delta_0$ sufficiently small this defines a smooth parametrization of the $\delta_0$-tubular neighbourhood of $S$. We then define,
for $|L_\pm(\cdot,\eps)|<\delta_0$, the sets 
\begin{align*}
	U_\e \,&:=\, \Psi \big(\big\{(p,t)\,:\, p\in S,\, t\in (\ell_-(\eps),\ell_+(\eps))\big\}\big),\\
    V_\eps\,&:=\, \Psi \big(\big\{(p,t)\,:\, p\in S,\, t\in (L_-(p,\eps),\ell_-(\eps))\cup(\ell_+(\eps),L_+(p,\eps))\big\}\big)
\end{align*}
\begin{figure}
\centering
\includegraphics[width=0.65\textwidth]{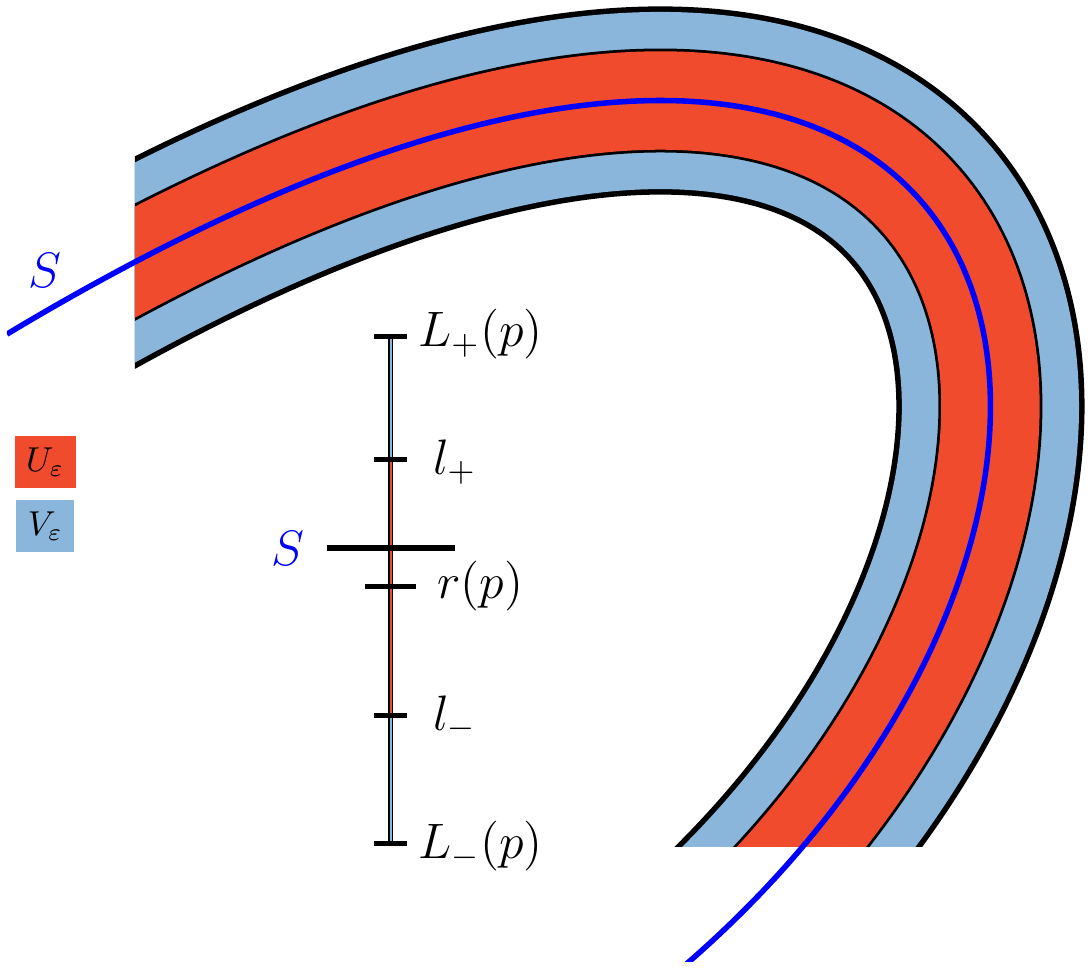}
\caption{Bilayer construction}
\label{fig:limsup}
\end{figure}
and corresponding rescaled density functions  
$$
	u_\e:=\frac{1}{\e}\chi_{U_\e}, \quad v_\e:=\frac{1}{\e}\chi_{V_\e}.
$$
Since for $\eps>0$ small enough $S$, $\nu$, and $L_\pm(\cdot,\eps)$ are smooth, and $\ell_\pm(\eps,\cdot)$ are constant, we obtain that 
$$
	(u_\e,v_\e)\in BV(\R^3;\{0,1/\e\}) \times L^1(\R^3;\{0,1/\e\})
$$
and that $u_\e v_\e=0$ almost everywhere in $\R^3$.

Note that the parametrization \eqref{eq:def-psi-limsup} is analogous to \eqref{eq:def-psi} (with $\theta$ replaced by $\nu$), but that the `reference surface' $\Psi(\{t=0\})$ does not describe the boundary of the $u_\eps$ support. Therefore also the roles of $\ell_\pm$ and $L_\pm$ are different from those in the previous section. Nevertheless we can introduce mass coordinates $\mathfrak{m}_p:\R\to \R$, $p\in S$, as in \eqref{eq:def-mass}, and obtain by analogous calculations as before
\begin{align}\label{eq:def-mass-limsup}
  \mathfrak{m}(p,t) &:=
   \frac{1}{\eps}\bigg(t+\frac{t^2}{2}H(p)+\frac{t^3}{3}K(p)\bigg),\end{align}
where $H$ and $K$ are the scalar total and Gaussian curvature of $S$ (with respect to $\nu$).
In particular we can express the total mass constraint for $u_\eps$ as 
\begin{align}
	1\,&=\, \int_{\R^3} u_\eps \,=\, \int_S \big(\mathfrak{m}_p(\ell_+(\e))-\mathfrak{m}_p(\ell_-(\e))\big)\,d\mathcal H^2(p), \notag\\
	&=\, \frac{1}{2\eps}(\ell_+(\eps)-\ell_-(\eps)) +\frac{1}{2\eps}(\ell_+(\eps)^2-\ell_-(\eps)^2)\int_S H \,d\Ha^2 + \frac{1}{3\eps}(\ell_+(\eps)^3-\ell_-(\eps)^3)\int_S K \,d\Ha^2. \label{cond1}
\end{align}
In order to also satisfy the mass constraint for $v_\eps$ we require
\begin{equation}\label{cond2}
	\mathfrak{m}_p(L_+(\e,p))-\mathfrak{m}_p(L_-(\e,p))=2(\mathfrak{m}_p(\ell_+(\e))-\mathfrak{m}_p(\ell_-	(\e))) \quad \forall p\in S,
\end{equation}
which then implies together with \eqref{cond1}
\begin{align*}
	\int_{\R^3} v_\eps \,&=\, \int_S \big( \mathfrak{m}_p(L_+(\e,p))-\mathfrak{m}_p(\ell_+(\e)) + \mathfrak{m}_p(\ell_-(\e))-\mathfrak{m}_p(L_-(\e,p))\big)\,d\Ha^2(p)\\
	&=\,  \int_S \big(\mathfrak{m}_p(\ell_+(\e))-\mathfrak{m}_p(\ell_-(\e))\big)\,d\mathcal H^2(p)\,=\, 1.
\end{align*}
Thus, $(u_\e,v_\e)\in\mathcal K_\e$ if and only if \eqref{cond1}, \eqref{cond2} hold.

This directly gives an estimate for the surface term of the energy.
Since $J_{u_\eps}\,=\, \{\Psi(p,\ell_+(\eps))\,:\, p\in S\}\cup \{\Psi(p,\ell_-(\eps))\,:\, p\in S\}$ and since we have for the Jacobian of $\Psi$ that $J\Psi(p,t)\,=\, 1+ H(p)t + K(p)t^2$, we deduce that
\begin{equation}\label{eq:Ju}
	\mathcal H^2(J_{u_\e}) \,=\, 1+(\ell_+(\e)+\ell_-(\e))\int_SH\,d\mathcal H^2+(\ell_+^2(\e)+\ell_-^2(\e))\int_SK\,d\mathcal H^2.
\end{equation}

\subsection{One-dimensional transport along orthogonal rays}
We next turn to an estimate for the distance term and first define on each orthogonal ray a point $p+r_{p,\e}\nu(p)$ that separates the part of the ray where mass is transported in the positive direction from that where mass is transported in the negative direction of $\nu(p)$. We observe that 
\begin{align*}
	\mathfrak{m}_p(\ell_-(\eps)) - \big(2\mathfrak{m}_p(\ell_+(\eps))-\mathfrak{m}_p(L_+(\eps,p))\big)\,=\, -\mathfrak{m}_p(\ell_-(\eps)) + \mathfrak{m}_p(L_-(\eps,p)) \,&<\, 0,\\
	\mathfrak{m}_p(\ell_+(\eps)) - \big(2\mathfrak{m}_p(\ell_+(\eps))-\mathfrak{m}_p(L_+(\eps,p))\big)\,=\, -\mathfrak{m}_p(L_+(\eps,p))-\mathfrak{m}_p(\ell_+(\eps))  \,&>\, 0.
\end{align*}
We therefore deduce that there is a unique $r_{p,\e}\in (\ell_-(\e),\ell_+(\e))$ such that
\begin{equation}\label{eq:rp}
	\mathfrak{m}_p(L_+(\e,p))-\mathfrak{m}_p(\ell_+(\e)) \,=\, \mathfrak{m}_p(\ell_+(\e))-\mathfrak{m}_p(r_{p,\e})
\end{equation}
is satisfied. 
As in Proposition \ref{prop-mass} we observe that $\mathfrak{m}_p$ is invertible, with inverse denoted by $\mathfrak{t}_p$. Therefore, $r_{p,\e}$ is given by
$$
	r_{p,\e} \,=\, \mathfrak{t}_p(2\mathfrak{m}_p(\ell_+(\e))-\mathfrak{m}_p(L_+(\e,p))).
$$
Together with \eqref{cond2} we deduce from \eqref{eq:rp} that
\begin{equation}\label{eq:rp1}
\mathfrak{m}_p(r_{p,\e})-\mathfrak{m}_p(\ell_-(\e)) \,=\, \mathfrak{m}_p(\ell_-(\e))-\mathfrak{m}_p(L_-(\e,p)).
\end{equation}
For each $p\in S$, each $\e$ sufficiently small and each $t\in (\ell_-(\e),\ell_+(\e))$ we define a map $\phi_p$ by requiring
\begin{equation}\label{eq:r1}
	\mathfrak{m}_p(\ell_+(\e))-\mathfrak{m}_p(t)\,=\, \mathfrak{m}_p(L_+(\e,p))-\mathfrak{m}_p(\phi_p(t)), \quad 	\textrm{for $t\in (r_{p,\e},\ell_+(\e))$},
\end{equation}
and
\begin{equation}\label{eq:r2}
\mathfrak{m}_p(t)-\mathfrak{m}_p(\ell_-(\e))=\mathfrak{m}_p(\phi_p(t))-\mathfrak{m}_p(L_-(\e,p)), \quad \textrm{for $t\in (\ell_-(\e),r_{p,\e})$}.
\end{equation}
This yields
$$
	\phi_p(t)\,=\,
	\begin{cases}
		\mathfrak{t}_p\big(\mathfrak{m}_p(L_+(\e,p))-\mathfrak{m}_p(\ell_+(\e))+\mathfrak{m}_p(t)\big) &\text{ if }t\in (r_{p,\e},\ell_+(\e)),\\
		\mathfrak{t}_p\big(\mathfrak{m}_p(L_-(\e,p))-\mathfrak{m}_p(\ell_-(\e))+\mathfrak{m}_p(t)) &\text{ if }t\in (\ell_-(\e),r_{p,\e})
	\end{cases}
$$
and in particular
\begin{align}
	\phi_p(\mathfrak{t}_p(m))=\left\{\begin{array}{ll}\mathfrak{t}_p(\mathfrak{m}_p(L_+(\e,p))-\mathfrak{m}_p(\ell_+(\e))+m), &  \textrm{if $m\in (\mathfrak{m}_p(r_{p,\e}),\mathfrak{m}_p(\ell_+(\e)))$}\\
\\
	\mathfrak{t}_p(\mathfrak{m}_p(L_-(\e,p))-\mathfrak{m}_p(\ell_-(\e))+m), & \textrm{if $m\in (\mathfrak{m}_p(\ell_-(\e)),\mathfrak{m}_p(r_{p,\e}))$}.
\end{array}\right. \label{eq:phi-mass}
\end{align}

\subsection{Three-dimensional transport}
We then claim that
\begin{align*}
	\Phi: U_\eps\,\to\, V_\eps,\quad p+t\nu_p \,\mapsto\, p+\phi_p(t)\nu_p,
\end{align*}
is a transport map from $u_\eps$ to $v_\eps$. In fact, by applying a transformation formula as in \eqref{eq:trafo2} we deduce that for any $g\in L^1(\R^3)$
\begin{eqnarray*}
	\lefteqn{\int g(\Phi(x)) u_\eps(x)\,dx}&& \\
	&= &\int_S \int_{\mathfrak{m}_p(\ell_-(\e))}^{\mathfrak{m}_p(r_{p,\e})}g\bigl[ p+\phi_p(\mathfrak{t}_p(m))\nu_p\bigr]\,dm + \int_S \int_{\mathfrak{m}_p(r_{p,\e})}^{\mathfrak{m}_p(\ell_+(\e))}g\bigl[ p+\phi_p(\mathfrak{t}_p(m))\nu_p\bigr]\,dm\\
	&\stackrel{\eqref{eq:phi-mass}}= &\int_S \int_{\mathfrak{m}_p(\ell_-(\e))}^{\mathfrak{m}_p(r_{p,\e})}g\Bigl[ p+\mathfrak{t}_p\bigl(\mathfrak{m}_p(L_-(\e,p))-\mathfrak{m}_p(\ell_-(\e))+m\bigr)\nu_p\Bigr]\,dm \\
	&&
	\quad + \int_S \int_{\mathfrak{m}_p(r_{p,\e})}^{\mathfrak{m}_p(\ell_+(\e))}g\Bigl[ p+\mathfrak{t}_p\bigl(\mathfrak{m}_p(L_+(\e,p))-\mathfrak{m}_p(\ell_+(\e))+m\bigr)\nu_p\Bigr]\,dm\\
	&\stackrel{(\ref{eq:rp}-\ref{eq:rp1})}= &\int_S \int_{\mathfrak{m}_p(L_-(\e))}^{\mathfrak{m}_p(\ell_-(\e))} g( p+\mathfrak{t}_p(m)\nu_p)\,dm + \int_S \int_{\mathfrak{m}_p(\ell_+(\e))}^{\mathfrak{m}_p(L_+(\e))}g( p+\mathfrak{t}_p(m)\nu_p)\,dm\\
	&=& \int g(x) v_\eps(x)\,dx.
\end{eqnarray*}
By the definition of $d_1$ this implies the estimate
$$
\begin{aligned}
d_1(u_\e,v_\e)&\le\int_S\int_{\mathfrak{m}_p(r_{p,\e})}^{\mathfrak{m}_p(\ell_+(\e))}\bigl[\phi_p(\mathfrak{t}_p(m))-\mathfrak{t}_p(m)\bigr]\,dm\,d\mathcal H^2(p)\\
&\qquad +\int_S\int_{\mathfrak{m}_p(\ell_-(\e))}^{\mathfrak{m}_p(r_{p,\e})}\bigl[\mathfrak{t}_p(m)-\phi_p(\mathfrak{t}_p(m))\bigr]\,dm\,d\mathcal H^2(p).
\end{aligned}
$$
In particular, from \eqref{eq:phi-mass} we deduce that
\begin{equation}\label{eq:d1}
\begin{aligned}
d_1(u_\e,v_\e)&\le\int_S\bigg(\int_{\mathfrak{m}_p(r_{p,\e})}^{\mathfrak{m}_p(\ell_+(\e))}\bigl[\mathfrak{t}_p\bigl(\mathfrak{m}_p(L_+(\e,p))-\mathfrak{m}_p(\ell_+(\e))+m\bigr)-\mathfrak{t}_p(m)\bigr]\,dm\\
&+\int_{\mathfrak{m}_p(\ell_-(\e))}^{\mathfrak{m}_p(r_{p,\e})}\bigl[\mathfrak{t}_p(m)-\mathfrak{t}_p\bigl(\mathfrak{m}_p(L_-(\e,p))-\mathfrak{m}_p(\ell_-(\e))+m\bigr)\bigr]\,dm\bigg)d\mathcal H^2(p).
\end{aligned}
\end{equation}
For the sake of simplicity we use the abbrevations
$$
\delta_+:=\mathfrak{m}_p(L_+(\e,p))-\mathfrak{m}_p(\ell_+(\e)), \quad \delta_-:=\mathfrak{m}_p(L_-(\e,p))-\mathfrak{m}_p(\ell_-(\e))
$$
By \eqref{eq:dev-t-0} we obtain that $\mathfrak{t}_p$ expands as 
$$
\mathfrak{t}_p(m)=\e m-\frac{\e^2}{2}H(p)m^2+\frac{\e^3}{6}(3H^2(p)-2K(p))m^3+o(\e^3),
$$
and a straightforward computation yields that
\begin{equation}\label{eq:int1}
\begin{aligned}
\mathfrak{t}_p(\delta_++m)-\mathfrak{t}_p(m)&=\e\delta_+-\frac{\e^2}{2}H(p)\delta_+^2-\e^2H(p)\delta_+m+\frac{\e^3}{6}(3H^2(p)-2K(p))\delta_+^3\\
&+\frac{\e^3}{2}(3H^2(p)-2K(p))\delta_+^2m+\frac{\e^3}{2}(3H^2(p)-2K(p))\delta_+m^2+o(\e^3)
\end{aligned}
\end{equation}
and
\begin{equation}\label{eq:int2}
\begin{aligned}
\mathfrak{t}_p(m)-\mathfrak{t}_p(\delta_-+m)&=-\e\delta_-+\frac{\e^2}{2}H(p)\delta_-^2+\e^2H(p)\delta_-m-\frac{\e^3}{6}(3H^2(p)-2K(p))\delta_-^3\\
&-\frac{\e^3}{2}(3H^2(p)-2K(p))\delta_-^2m-\frac{\e^3}{2}(3H^2(p)-2K(p))\delta_-m^2+o(\e^3).
\end{aligned}
\end{equation}
Putting \eqref{eq:int1} and \eqref{eq:int2} in \eqref{eq:d1} and integrating in $dm$ we easily obtain
\begin{equation}
\begin{aligned}\label{eq:d1final}
d_1(u_\e,v_\e)&\le \e \int_S\Bigl[\bigl(\mathfrak{m}_p(L_+(\e,p))-\mathfrak{m}_p(\ell_+(\e))\bigr)^2+\bigl(\mathfrak{m}_p(L_-(\e,p))-\mathfrak{m}_p(\ell_-(\e))\bigr)^2\Bigr]\,d\mathcal H^2(p)\\
&-\e^2\int_S\Bigl[\mathfrak{m}_p(\ell_+(\e))\bigl(\mathfrak{m}_p(L_+(\e,p))-\mathfrak{m}_p(\ell_+(\e))\bigr)^2\\
&\qquad \qquad +\mathfrak{m}_p(\ell_-(\e))\bigl(\mathfrak{m}_p(L_-(\e,p))-\mathfrak{m}_p(\ell_-(\e))\bigr)^2)\Bigr]H(p)\,d\mathcal H^2(p)\\
&+\frac{\e^3}{12}\int_S\Bigl[(\mathfrak{m}_p(\ell_+(\e))-\mathfrak{m}_p(r_{p,\e}))^2\bigl(7\mathfrak{m}_p(\ell_+(\e))^2+\mathfrak{m}_p(r_{p,\e})^2-2\mathfrak{m}_p(\ell_+(\e))\mathfrak{m}_p(r_{p,\e})\bigr)\\
&\qquad  +(\mathfrak{m}_p(r_{p,\e})-\mathfrak{m}_p(\ell_-(\e)))^2\bigl(7\mathfrak{m}_p(\ell_-(\e))^2+\mathfrak{m}_p(r_{p,\e})^2-2\mathfrak{m}_p(\ell_-(\e))\mathfrak{m}_p(r_{p,\e})\bigr)\Bigr]\\
&\qquad \qquad \qquad (3H^2(p)-2K(p))\,d\mathcal H^2(p)+o(\e^3).
\end{aligned}
\end{equation}

\subsection{Finalizing the choices}
With these preparations one now can estimate the values of $\Gfun_\e(u_\eps,v_\eps)$ and optimize this estimate with respect to the choice of $\ell_\pm(\eps)$. For simplicity we do it here the other way around and start with a suitable choice:  For any $\e>0$ we let 
$$
	\ell_+(\e):=\e+a(\e)\e^3, \quad \ell_-(\e):=-\ell_+(\e)=-\e-a(\e)\e^3,
$$
where $\e \mapsto a(\e)$ is the unique smooth map, defined for $\e$ sufficiently small, such that 
\begin{gather}\label{implicitell}
	a(\e)+\frac{2}{3} \int_SK\,d\mathcal H^2+\bigg(2a(\e)\e^2+2a^2(\e)\e^4+\frac{2\e^6}{3}a^3(\e)\bigg)\int_SK\,d\mathcal H^2 \,=\, 0, \\
	a(0)=-\frac{2}{3}\int_SK\,d\mathcal H^2,
\end{gather}
where we remark that \eqref{implicitell} is equivalent to the total mass condition \eqref{cond1}.
%Thus 
%$$
%a(\e)=-\frac{2}{3}\int_SK\,d\mathcal H^2+o(\e)
%$$
%so that 
%$$
%\ell_\pm(\e)=\pm\e\mp\frac{2\e^3}{3}\int_SK\,d\mathcal H^2+o(\e^4).
%$$
%We easily obtain, from the implicit relation in \eqref{implicitell} and using the very definition of $\mathfrak{m}_p$, 
%$$
%\int_S(\mathfrak{m}_p(\ell_+(\e))-\mathfrak{m}_p(\ell_-(\e)))\,d\mathcal H^2(p)=1
%$$
%which is \eqref{cond1}. Moreover, since it holds
Since
$$
\ell_\pm(\e)=\pm\e\mp\frac{2\e^3}{3}\int_SK\,d\mathcal H^2+o(\e^4),
$$
the expression~\eqref{eq:Ju} for the surface area becomes
\begin{equation}\label{eq:Jufinal}
\mathcal H^2(J_{u_\e})=1+2\e^2\int_SK\,d\mathcal H^2+o(\e^2).
\end{equation}
Next, we pass to the estimate of $d_1(u_\e,v_\e)$. First of all note that
$$
\mathfrak{m}_p(\ell_+(\e))=1+\frac{\e}{2}H(p)+\e^2\bigg(a(\e)+\frac{1}{3}K(p)\bigg)+o(\e^2),
$$
while
$$
\mathfrak{m}_p(\ell_-(\e))=-1+\frac{\e}{2}H(p)-\e^2\bigg(a(\e)+\frac{1}{3}K(p)\bigg)+o(\e^2).
$$
Now, fix $p\in S$ and consider the unique smooth map $\e \mapsto L_+(\e,p)$, such that
$$
\mathfrak{m}_p(L_+(\e,p))=2+\frac{3\e}{2}H(p), \quad L_+(0,p)=0.
$$
Therefore, we are able to find a unique smooth map $\e \mapsto L_-(\e,p)$ such that
$$
\mathfrak{m}_p(L_+(\e,p))-\mathfrak{m}_p(L_-(\e,p))=2(\mathfrak{m}_p(\ell_+(\e))-\mathfrak{m}_p(\ell_-(\e)))
$$
which is \eqref{cond2}, and from which we get 
$$
\mathfrak{m}_p(L_-(\e,p))=-2+\frac{3\e}{2}H(p)-4\e^2\bigg(a(\e)+\frac{1}{3}K(p)\bigg)+o(\e^2).
$$
Finally, we have, recalling \eqref{eq:rp},  
$$
\mathfrak{m}_p(r_{p,\e})=-\frac{\e}{2}H(p)+2\e^2\bigg(a(\e)+\frac{1}{3}K(p)\bigg)+o(\e^2).
$$
%Using now the very definition of $\delta_+$ and $\delta_-$ we obtain 
%$$
%\delta_+=1+\e H(p)-\e^2\bigg(a(\e)+\frac{1}{3}K(p)\bigg)+o(\e^2), 
%$$
%and
%$$
%\delta_-=-1+\e H(p)-\e^2(3a(\e)+K(p))+o(\e^2).
%$$
%Thus, by \eqref{eq:int1}, for each $m\in (\mathfrak{m}_p(r_{p,\e}),\ell_+(\e))$ we obtain
%$$
%\mathfrak{t}_p(\delta_++m)-\mathfrak{t}_p(m)=
%$$
Putting all together in \eqref{eq:d1final} we deduce that
$$
\begin{aligned}
d_1&(u_\e,v_\e)\\
&\le \e+2\e^3\int_S\bigg(H^2+2a(\e)+\frac{2}{3}K\bigg)\,d\mathcal H^2-5\e^3\int_SH^2\,d\mathcal H^2+\frac{7\e^3}{6}\int_S(3H^2-2K)\,d\mathcal H^2+o(\e^3)\\
&\qquad =\e+\e^3\int_S\bigg(\frac{1}{2}H^2-\frac{7}{3}K\bigg)\,d\mathcal H^2+o(\e^3).
\end{aligned}
$$
Finally, together with \eqref{eq:Jufinal} we get
$$
\mathcal H^2(J_{u_\e})+\frac{1}{\e}d_1(u_\e,v_\e) \le 2+2\e^2\int_S \bigg(\frac{1}{4}H^2-\frac{1}{6}K\bigg)\,d\mathcal H^2+o(\e^2)
$$
which gives \eqref{eq:limsup}, by the definition of $\mathcal G_\e(u_\e,v_\e)$.
\\
\par 
It remains to prove that $u_\e\mathcal L^3 \stackrel*\rightharpoonup2\mathcal H^2\lfloor S$ as measures. For any $\varphi\in C^0_c(\R^3)$ we have, using \eqref{eq:trafo}, 
$$
\int\varphi(x) u_\e(x)\,dx=\frac{1}{\e}\int_{U_\e}\varphi(x)\,dx=\frac{1}{\e}\int_S\int_{\ell_-(\e)}^{\ell_+(\e)}\varphi(p+t\nu(p))(1+tH(p)+t^2K(p))\,d\mathcal H^2(p)
$$
from which it follows, since $\ell_\pm(\e)=\pm \e+o(\e)$, that
$$
\lim_{\e\to 0}\int\varphi(x) u_\e(x)\,dx=2\int_S\varphi(p)\,d\mathcal H^2(p)
$$
and this yields the conclusion.
%
%=======================================
%
% ==========================================
%\bibliographystyle{abbrv}\bibliography{all,eigene}%,all
\def\cprime{$'$}

\end{document}